\documentclass[11pt,leqno]{amsart}

\usepackage{amsmath}
\usepackage{amssymb}
\usepackage{a4wide}
\usepackage{graphics}
\usepackage{epsfig}
\usepackage{mathrsfs}
\usepackage{enumerate}
\usepackage{cancel}
\usepackage{graphicx}
\usepackage{hyperref}
\hypersetup{colorlinks,%
	citecolor=cyan,%
	filecolor=red,%
	linkcolor=red,%
	urlcolor=red}
\newtheorem{theorem}{Theorem}
\newtheorem{lemma}[theorem]{Lemma}

\newtheorem{remark}[theorem]{Remark}

\newcounter{hypo}

\def\C{{\mathbb C}}

\def\R{{\mathbb R}}

\def\S{{\mathbb S}}

\def\CE{\mathcal {E}}

\def\CH{\mathcal {H}}

\def\CK{\mathcal {K}}

\def\CO{\mathcal {O}}

\def\CR{\mathcal {R}}

\def\CX{\mathcal {X}}

\def\p{\partial}

\def\ker{\mathop{\rm Ker}\nolimits}

\def\<{\langle}
\def\>{\rangle}

\def\square{\hbox{\vrule\vbox{\hrule\phantom{o}\hrule}\vrule}}

\newcommand\pa{\partial}

\renewcommand\Box{{\square}}

\newcommand\cO{\mathcal{O}}

\newcommand\cR{\mathcal{R}}

\newcommand\cK{\mathcal{K}}

\newcommand\cX{\mathcal{X}}

\newtheorem*{thm*}{Theorem}
\newtheorem*{prop*}{Proposition}
\newtheorem*{cor*}{Corollary}
\newtheorem*{conj*}{Conjecture}
\theoremstyle{remark}
\newtheorem*{rem*}{Remark}
\theoremstyle{definition}
\newtheorem*{Def*}{Definition}

\numberwithin{theorem}{section}
\numberwithin{equation}{section}

\author[]{Nicolas Besset and Dietrich H\"afner}
\address{Universit\'e Grenoble Alpes, Institut
	Fourier, UMR 5582 du CNRS, 100 rue des maths, 38610 Gi\`eres, France}
\email{Nicolas.Besset@univ-grenoble-alpes.fr, Dietrich.Hafner@univ-grenoble-alpes.fr}

\title[Exponentially growing modes]{Existence of exponentially growing finite energy solutions for the charged Klein-Gordon equation on the De Sitter-Kerr-Newman metric}


\subjclass[2000]{}
\begin{document}
	\begin{abstract}
		\centering We show the existence of exponentially growing finite energy solutions for the charged Klein-Gordon equation on the De Sitter-Kerr-Newman metric for small charge and mass of the field and small angular momentum of the black hole. The mechanism behind is that the zero resonance that exists for zero charge, mass and angular momentum moves into the upper half plane. 
	\end{abstract}
	\maketitle
	\section{Introduction}
	In this note we show the existence of exponentially growing modes for the massive charged Klein-Gordon equation on the De Sitter-Kerr-Newman metric when the product of the charge of the black hole and the charge of the particle is small, and the angular momentum of the black hole and the mass of the field are small with respect to the charge product. The cosmological constant is also supposed to be small. A calculation with physical constants shows that this smallness assumption is fulfilled for physical slowly rotating black holes.  A natural global energy associated to the charged Klein-Gordon equation on the De Sitter-Kerr-Newman exterior is finite, but exponentially growing. The origin of the growing mode is the strong coupling of the charges (with respect to the mass of the field) for non-rotating charged black holes \textit{i.e.} on the De Sitter-Reissner-Nordstr\"om metric, and the mode still exists for slowly-rotating backgrounds. This is in contrast to the case when the charge product is small with respect to the mass, in which no exponentially growing mode exists, see \cite{Be}. In the latter case the system can be understood as a small perturbation of the massive Klein-Gordon equation for which there exists a resonance free region ${\rm Im}\, \sigma>-\delta,\, \delta>0$. For small perturbations no resonance can move up to the real axis. In contrast to the Klein-Gordon equation, the wave equation has a zero resonance, see \cite{BoHa}. The present setting is a perturbation of this case. It turns out that for small perturbations the zero resonance moves into the upper half plane if the cosmological constant is small enough as shown in this paper. Note that this is different from the wave equation on the De Sitter-Kerr metric which can also be considered as a perturbation of the De Sitter-Schwarzschild case, but for which the zero resonance persists, see \cite{Dy}. Our results are consistent with the numerical results of \cite{CCDHJ}.  Existence of exponentially growing modes is known for the Klein-Gordon equation on the Kerr metric for certain masses, see \cite{SR}. We also refer to \cite{Mosch} for similar results for the wave equation with an additional potential. 
	  
	The De Sitter-Kerr-Newman metric defines a rotating charged black hole and is a solution of the Einstein-Maxwell system with positive cosmological constant. As mentioned above, the main result of this paper is essentially contained in the non-rotating case \textit{i.e.} for the De Sitter-Reissner-Nordstr\"om metric. The question of growing modes for hyperbolic equations on black hole type backgrounds is usually very much linked to the question of stability of this background itself as a solution of the Einstein equations. In the present case the instability comes from the charge of the particle which is not relevant for the stability question for the Einstein-Maxwell system. In particular the De Sitter-Kerr-Newman solution is known to be stable as a solution of the Einstein-Maxwell system with positive cosmological constant and a small angular momentum, see \cite{HinlsKN}. Our result nevertheless suggests a possible instability of the coupled Einstein-Maxwell-charged scalar field system.   
	
	There exists a huge literature concerning dispersive estimates for hyperbolic equations on black hole type spacetimes linked to the stability question, see \cite{HiVa} for the stability of the De Sitter-Kerr black hole for small angular momentum, \cite{KlSz} for the stability of the Schwarzschild black hole under polarized perturbations as well as references therein for an overview. 
	
	We first formulate and show our result only for the De Sitter-Reissner-Nordstr\"om black hole. This is technically easier and all central arguments are already contained in this part.  The existence of an exponentially growing mode is then stable with respect to a small perturbation of the spacetime and thus holds for the De Sitter-Kerr-Newman spacetime if the angular momentum is small with respect to the charge product. The decomposition we use for our perturbation arguments is inspired from \cite{HHV}.  
	
	The paper is organized as follows. In Section \ref{Sec2} we recall some essential features of the De Sitter-Reissner-Nordstr\"om solution.  In Section \ref{Sec3} we recall the Fredholm framework which goes essentially back to \cite{Va} and formulate the main result in the De Sitter-Reissner-Nordstr\"om case. Section \ref{Sec4} is devoted to the proof of the main theorem. Finally, we extend our result to the rotating case in Section \ref{Sec5}. 
	
	{\bf Acknowledgments.}
	
	We thank Peter Hintz for helpful comments on a preliminary version of this paper. D.H.\\ 
	acknowledges support from the ANR funding ANR-16-CE40-0012-01. He also thanks the MSRI in Berkeley and the Mittag-Leffler Institute in Stockholm for hospitality during his stays in autumn 2019. 
	
	\section{The De Sitter-Reissner-Nordstr\"om metric}
	\label{Sec2}
	In this section we closely follow \cite[Section 3]{HinlsKN}. Let $M>0$ be the mass of the black hole, $Q\in\mathbb{R}$ its electric charge and $\Lambda>0$ the cosmological constant. We will assume $|Q|<\frac{3M}{2\sqrt{2}}$ and $\Lambda$ sufficiently small such that
	\begin{equation*}
	  \mu(r):=1-\frac{2M}{r}+\frac{Q^{2}}{r^{2}}-\frac{\Lambda r^{2}}{3}
	\end{equation*}
	has four real roots $r_n<0<r_c<r_-<r_+$, and is positive on $(r_-,r_+)$ and that $\frac{\mu}{r^2}$ has a single non degenerate maximum on $(r_-,r_+)$, see \cite[Proposition 3.2]{HinlsKN} and \cite[Proposition 1]{Mo}. Note that for $Q=0$ the assumption is fulfilled for $9\Lambda M^2<1$.
	The (exterior) De Sitter-Reissner-Nordstr\"{o}m spacetime is the Lorentzian manifold $(\mathcal{M},g)$ with
	\begin{equation*}
	\mathcal{M}=\mathbb{R}_t\times(r_-,r_+)_r\times\mathbb{S}^2_\omega,\qquad\qquad g=\mu\mathrm{d}t^{2}-\mu^{-1}\mathrm{d}r^{2}-r^{2}\mathrm{d}\omega^{2} 
	\end{equation*}
	where $\mathrm{d}\omega^2$ is the standard metric on the unit sphere $\mathbb{S}^2$. We want to consider a charged scalar field on this spacetime with charge $q\in\R$. Let 
	\begin{equation*}
	A:=\frac{Q}{r}\mathrm{d}t,\, s:=qQ,\, V(r):=\frac{1}{r}.
	\end{equation*} 
	Note that $A$ comes from the Maxwell equation as the vector potential of the Coulomb solution. Then the charged wave operator on $(\mathcal{M},g)$ is
	\begin{equation*}
	\cancel{\Box}_{g}:=\left(\nabla_\alpha-iqA_\alpha\right)\left(\nabla^\alpha-iqA^\alpha\right)=\frac{1}{\mu}\left((\p_{t}-isV)^{2}-\frac{\mu}{r^{2}}\p_{r}r^{2}\mu\p_{r}-\frac{\mu}{r^{2}}\Delta_{\mathbb{S}^{2}}\right)
	\end{equation*}
	and the corresponding charged Klein-Gordon equation reads
	\begin{equation*}
	\cancel{\Box}_{g}u+m^{2}u=0
	\end{equation*}
	where $m>0$ is the mass of the particle.
	
	 Let $\mathfrak{r}\in(r_{-},r_{+})$ be the unique non-degenerate maximum of $\mu$ in $(r_-,r_+)$ (see \cite[Section 3]{HinlsKN} for the existence and uniqueness) and put $c^{2}:=\mu(\mathfrak{r})^{-1/2}$. We define a function $T\in C^{\infty}((r_-,r_+))$ such that 
	\[T'(r)=-\frac{\nu}{\mu},\]
	where  $\nu=\mp\sqrt{1-\mu c^{2}}$ for $\pm(r-\mathfrak{r})\geq 0$. Note that $\nu$ is smooth. We then put 
	\[t_*=t-T(r).\]
	By \cite[Lemma 3.3]{HinlsKN} the metric and its inverse read in the coordinates $(t_{*},r,\omega)\in\mathbb{R}\times(r_{-}-\eta_0,r_{+}+\eta_0)\times\mathbb{S}^{2}$ with $\eta_0>0$ small,
	\begin{equation*}
	g=\mu\mathrm{d}t_*^2-2\nu\mathrm{d}t_*\mathrm{d}r-c^2\mathrm{d}r^2-r^2\mathrm{d}\omega^2,\qquad\qquad g^{-1}=c^{2}\p_{t_*}^2-2\nu\p_{t_*}\p_{r}-\mu\p_{r}^2-r^{-2}e_{\mathbb{S}^{2}},
	\end{equation*}
	 where  $e_{\mathbb{S}^{2}}$ is dual to $\mathrm{d}\omega^{2}$. Observe that the coordinate transformation turns the electromagnetic potential $A$ to 
	\[A=\frac{Q}{r}\left(\mathrm{d}t_*-\frac{\nu}{\mu}\mathrm{d}r\right)=A_{t_*}\mathrm{d}t_*+A_r\mathrm{d}r.\]
        Notice that $A_r$ is singular at $r_\pm$. The charged Klein-Gordon operator now becomes:
	\begin{align*}
	\cancel{\Box}_{g}+m^2&=c^{2}(\p_{t_{*}}-iqA_{t_{*}})^{2}-\frac{1}{r^{2}}(\p_{t_{*}}-iqA_{t_{*}})r^{2}\nu(\p_{r}-iqA_{r})-\frac{1}{r^{2}}(\p_{r}-iqA_{r})r^{2}\nu(\p_{t_{*}}-iqA_{t_{*}})\\
	&-\frac{1}{r^{2}}(\p_{r}-iqA_{r})r^{2}\mu(\p_{r}-iqA_{r})-\frac{1}{r^{2}}\Delta_{\mathbb{S}^{2}_{\omega}}+m^2.
	\end{align*}
	
	We then observe that the charged Klein-Gordon equation is gauge invariant and that $\mathrm{d}\big(\frac{Q\nu}{\mu r}\mathrm{d}r\big)=0$. Thus, if $u$ is a solution to the charged Klein-Gordon equation with the above electromagnetic potential, then $e^{-iR}u$ with $R'(r)=\frac{Q\nu}{\mu r}$ is a solution of the Klein-Gordon equation with the (smooth) electromagnetic potential $A'=\frac{Q}{r}\mathrm{d}t_*$. This is equivalent to conjugating $\cancel{\Box}_{g}+m^{2}$ with $e^{-iR}$. We have:
	\begin{align*}
	re^{iR}(\cancel{\Box}_{g}+m^{2})e^{-iR}\frac{1}{r}&=c^{2}(\p_{t_{*}}-iqA_{t_{*}})^{2}-\frac{1}{r}(\p_{t_{*}}-iqA_{t_{*}})r^{2}\nu\p_r\frac{1}{r}-\frac{1}{r}\p_rr^{2}\nu(\p_{t_{*}}-iqA_{t_{*}})\frac{1}{r}\\
	&-\frac{1}{r}\p_rr^{2}\mu\p_r\frac{1}{r}-\frac{1}{r^{2}}\Delta_{\mathbb{S}^{2}_{\omega}}+m^{2}.
	\end{align*}
	Note that the operator is in particular conjugated by $r$. 
	We denote by $\hat{P}(\sigma,s,m)$ the Fourier transformed operator, i.e. 
	\begin{equation*}
	\hat{P}(\sigma,s,m)=-c^2(\sigma+sV)^2+\frac{i}{r}(\sigma+sV)r^2\nu\pa_r\frac{1}{r}+\frac{i}{r}\p_r r^2\nu(\sigma+sV)\frac{1}{r}-\frac{1}{r}\p_r\mu r^2\pa_r\frac{1}{r}-\frac{1}{r^{2}}\Delta_{\mathbb{S}^{2}_{\omega}}+m^{2}. 
	\end{equation*}	
	
	\section{Main result}
	\label{SecMainresult}
	\label{Sec3}
	\subsection{Function spaces}
	We choose $\eta_0>0$ sufficiently small such that $\mu$ has no root on $(r_--\eta_0,r_-)\cup(r_+,r_++\eta_0)$. Let $Y=(r_{-}-\eta_0,r_{+}+\eta_0)\times\mathbb{S}^2$ and let for $\gamma\in \R$, $\bar{H}^{\gamma}\left(Y;\mathrm{d}r\mathrm{d}\omega\right)$ be the Sobolev space of order $\gamma$ consisting of extendible distributions. We also put
	\[{\mathcal X}^{\gamma}=\big\{u\in \bar{H}^{\gamma}(Y)\,:\, \hat{P}(\sigma,s,m)u\in \bar{H}^{\gamma-1}(Y)\big\}.\]
	Observe that $\mathcal{X}^\gamma$ does not depend on $\sigma$, $s$ and $m$. We have indeed:
	\[\CX^{\gamma}=\{u\in \bar{H}^{\gamma}(Y)\,:\,\hat{P}(0,0,0)u\in \bar{H}^{\gamma-1}(Y)\}.\]
	We denote by $\<f,g\>=\frac{1}{4\pi}\int_Y \bar{f}g \mathrm{d}r\mathrm{d}\omega$ the inner product on $L^2(Y;\mathrm{d}r\mathrm{d}\omega)$. We then have (see \cite[Theorem 4.1]{HinlsKN}):
	\begin{theorem}
		\label{Fredholm}
		Let $\alpha>0$. Then there exists  $\gamma_0>1/2$ such that for all $\gamma>\gamma_0,\, s\in \R,$
		\begin{equation*}
		\hat{P}(\sigma,s,m)={\mathcal X}^{\gamma}\rightarrow \bar{H}^{\gamma-1}(Y),\qquad{\rm Im} \sigma>-\alpha
		\end{equation*}
		is a holomorphic family of Fredholm operators. The inverse exists for ${\rm Im}\, \sigma$ large enough. In particular the index is $0$ for all ${\rm Im}\, \sigma>-\alpha$.
	\end{theorem}
	We fix $\gamma>5/2$ in the following.  
	\subsection{Main Theorem}
	\label{MainTheorem}
	The main result of this note is the following one :
	\begin{theorem}
	\label{MainThm}
		Fix $M>0$ and $Q$ such that $0<|Q|<\frac{3M}{2\sqrt{2}}$. There exists $\Lambda_0>0$ such that for all $\Lambda\in(0,\Lambda_0)$ there exists a constant $c_0(\Lambda)>0$ such that the following holds. Suppose $m^2\le c_0(\Lambda)q^2$. Then for $|q|>0$ sufficiently small, there exists $\sigma\in \C$ with ${\rm Im}\, \sigma\geq c_1 q^2>0$ such that 
		\[{\rm Ker}\,\hat{P}(\sigma,s,m)\cap {\rm Ker}(\Delta_{\mathbb{S}_{\omega}^2})\neq\{0\}.\]
	\end{theorem}
	\begin{remark}
	\begin{enumerate}
		
		\item There exists $s_0>0$ such that for all $\vert s\vert\le s_0$ and all $\ell\geq 1,$ 
		\[{\rm Ker}\, \hat{P}^{\ell}(\sigma,s,m)\cap \{\sigma\in \C\,:\,{\rm Im}\,\sigma>-\delta\}=\{0\}\]
		for some $\delta>0$ and $\hat{P}^{\ell}(\sigma,s,m)$ the restriction of $\hat{P}(\sigma,s,m)$ on $\ker\big(\Delta_{\mathbb{S}^{2}_{\omega}}+\ell(\ell+1)\big)$. This can be shown in the same way as in \cite{Be}. 
		\item The condition on the size of $\Lambda$ can be made more explicit. It is for example sufficient that 
		\[r_+^4+r_-^4-26r_+^2r_-^2>0,\quad r_+^2-2r_+r_--r_-^2>0,\]
		which is clearly fulfilled for $\Lambda$ small enough ($r_+\rightarrow \infty$ and $r_-$ remains bounded when $\Lambda\rightarrow 0$). We refer to Subsection \ref{Zeros subsection} for details.
		\item We expect that no such growing modes exist for large values of $\Lambda$. However for all physically realistic black holes, the conditions in $(ii)$ are fulfilled.
		We refer to Remark \ref{remLambdalarge} for details.     
	\end{enumerate}
\end{remark}
\subsection{Finite energy}
		Let $u\in{\rm Ker}\,\hat{P}(\sigma,s,m)\cap {\rm Ker}(\Delta_{\mathbb{S}_{\omega}^2})$. The function $v(t_*,r)=\frac{e^{-iR}}{r}e^{-i\sigma t_*}u(r)$ solves
		\[\left(\,\cancel{\Box}_{g}+m^{2}\right)v=0.\]
		Natural energies associated to the $t_*$ foliation are finite, but exponentially growing in $t_*$, e.g. 
		\[\CH(v)=\Vert(\p_{t_*}-iqA_{t_*})v\Vert^2+\Vert v\Vert^2_{\bar{H}^1}.\]
		It is also interesting to go back to the Boyer-Lindquist coordinates and the original gauge. We find 
		\[v(t,r)=e^{-i\sigma t}e^{i\sigma T(r)}\frac{e^{-iR(r)}}{r}u(r)=e^{-i\sigma t}w(r)\] 
		with \[w\in e^{(\epsilon-{\rm Im}\, \sigma)|T(r)|}L^2\left((r_-,r_+)\times\mathbb{S}^2;\frac{\mathrm{d}r}{\mu}\mathrm{d}\omega\right)\]
		for all $\epsilon>0$.
		
		The natural conserved energies associated to solutions of the charged Klein-Gordon equation are 
		\[\CE_l(v)=\Vert (\p_t-il)v\Vert^2+((\mu\hat{P}(0,0,m)-(sV-l)^2)v,v)\]
		for all $l\in\R$; we can check that none of them is positive. Here $(.,.)$ is the scalar product in $L^2\left((r_-,r_+)\times\mathbb{S}^2;\frac{\mathrm{d}r}{\mu}\mathrm{d}\omega\right)$. Following \cite{GGH} and \cite{Be}, we associate to the constructed solution $v=e^{-i\sigma t}w$ the positive global energy
		\[\dot{\CE}(v)=e^{2({\rm Im\,}\sigma)t}\left(\Vert (\sigma+sV)w\Vert^2+(\mu\hat{P}(0,0,m)w,w)\right)\]
		which is finite but exponentially growing in time. Note that the energy $\dot{\CE}$ is conserved for $s=0$.  
		\begin{remark}
		In the Boyer Lindquist picture we can apply directly  \cite[Proposition 3.6]{GGH} to see that all possible eigenvalues $\sigma$ with ${\rm Im}\, \sigma>0$ have to lie in the disc $\bar{D}(0,\vert s\vert\|V\|_{\infty}).$
		\end{remark}
		\section{Perturbation theory}
	\label{Sec4}
	We will decompose the function spaces with respect to the kernel and cokernel of $\hat{P}(0,0,0)$. 
	\subsection{Kernel and cokernel of $\hat{P}(0,0,0)$.}
	\begin{lemma}
		\label{lemma3}
		For $\ell\neq 0$, ${\rm Ker}\hat{P}^{\ell}(0,0,0)=\{0\}$, where $\hat{P}^{\ell}(0,0,0)$ denotes the restriction of $\hat{P}(0,0,0)$ on $\ker\big(\Delta_{\mathbb{S}^{2}_{\omega}}+\ell(\ell+1)\big)$. 
	\end{lemma}
	\begin{proof}
		Suppose $\hat{P}^{\ell}(0,0,0)u_{\ell}=0$. We multiply by $\bar{u}_{\ell}$ and integrate by parts over $[r_-,r_+]$:
		\[0=\int_{r_-}^{r_+}\mu r^2\left\vert \p_r\frac{1}{r}u_{\ell}\right\vert^2\mathrm{d}r\mathrm{d}\omega+\int_{r_-}^{r_+}\ell(\ell+1)\vert u_{\ell}\vert^2\mathrm{d}r\mathrm{d}\omega.\]
		It follows $u\equiv 0$ on $[r_-,r_+]$. We can then apply  \cite[Lemma 1]{Zw} to conclude. 
	\end{proof}
	\begin{lemma}
		\label{Kernel, cokernel}
		We have
		\[{\rm Ker}\hat{P}(0,0,0)=\C r:=\cK,\qquad{\rm Ker}\hat{P}^*(0,0,0)=\C r^*:=\cK^*,\qquad r^*={\bf 1}_{(r_-,r_+)}(r)r.\] 
	\end{lemma}
	\begin{proof}
		By Lemma \ref{lemma3}, we know that elements $u$ of the kernel satisfy $-\Delta_{\mathbb{S}^{2}_{\omega}}u=0$. Integration of $\hat{P}(0,0,0)u=0$ over $[r_0,r]$ for some $r_0\in(r_-,r_+)$ gives 
		\[u=r\int_{r_0}^r\frac{c_1}{\mu r^2}\mathrm{d}r'+c_2r\]
		with integration constants $c_1,\, c_2$. Thus $c_1=0$ because the first term on the R.H.S. above is not continuous at $r=r_{\pm}$ for $c_1\neq 0$, so that the kernel is $\C r$. By the index $0$ property, the dimension of the kernel of $\hat{P}^*(0,0,0)$ is also $1$, we therefore only have to check that $r^*\in {\rm Ker}\,\hat{P}^*(0,0,0)$. For all $u\in\cX^\gamma$, we find 
		\[\left\langle r^*,\hat{P}(0,0,0)u\right\rangle=-\frac{1}{4\pi}\int_{\S^2}\left(\int_{r_{-}}^{r_{+}}\p_r\mu r^2\p_r\frac{1}{r}u\mathrm{d}r\right)\mathrm{d}\omega-\frac{1}{4\pi}\int_{r_{-}}^{r_{+}}\frac{1}{r}\left(\int_{\S^2}\Delta_{\mathbb{S}^{2}}u\mathrm{d}\omega\right)\mathrm{d}r=0.\]
		This completes the proof.
	\end{proof}
	\subsection{Decomposition}
	Recall that $\CX^{\gamma}$ is independent of $\sigma$, $s$ and $m$. We decompose 
	\[\CX^{\gamma}=\CK^{\perp}\oplus\cK.\]
	Let $\cR$ be the closed subspace of $\bar{H}^{\gamma-1}(Y)$ such that 
	\[\<f,r^*\>=0,\quad\forall f\in \cR\]
	and $\cR^{\perp}=\C r$. Then we have 
	\[\bar{H}^{\gamma-1}(Y)=\CR\oplus \cR^{\perp}.\]
	Indeed $\<r,r^*\>\neq 0$ and $f\in \bar{H}^{\gamma-1}(Y)$ writes 
	\[f=\frac{\<f,r^*\>}{\<r,r^*\>}r+f-\frac{\<f,r^*\>}{\<r,r^*\>}r,\quad \left\<f-\frac{\<f,r^*\>}{\<r,r^*\>}r,r^*\right\>=0.\]
	Within this decomposition we can consider $\hat{P}(\sigma,s,m)$ as a matrix operator 
	\begin{equation*}
	\hat{P}(\sigma,s,m)=\left(\begin{array}{cc} \hat{P}_{00}(\sigma,s,m) & \hat{P}_{01}(\sigma,s,m)\\ \hat{P}_{10}(\sigma,s,m) & \hat{P}_{11}(\sigma,s,m) \end{array}\right):{\mathcal K}^{\perp}\oplus {\mathcal K}\rightarrow{\mathcal R}\oplus {\mathcal R}^{\perp}.
	\end{equation*}
	Note that $\hat{P}_{00}(\sigma,s,m)$ is invertible for $(\sigma,s,m)$ in a small neighborhood of $(0,0,0)$. This follows from the fact that $(\sigma,s,m)\mapsto \hat{P}_{00}(\sigma,s,m)$ is an analytic family of operators and that $\hat{P}_{00}(0,0,0)$ is invertible. We then obtain 
	 \[\hat{P}_{00}^{-1}(\sigma,s,m)=\hat{P}_{00}^{-1}(0,0,0)+\CO(\vert s \vert).\]
	 The crucial point is that $\hat{P}$ is invertible if and only if $\mathbb{P}:=\hat{P}_{11}-\hat{P}_{10}\hat{P}_{00}^{-1}\hat{P}_{01}:{\cK}\to\cR^\perp$ is. Indeed, if $\mathbb{P}$ is invertible, then an explicit calculation gives 
	\[\hat{P}^{-1}(\sigma)=\left(\begin{array}{cc} \hat{P}_{00}^{-1}({\bf 1}+\hat{P}_{01}\mathbb{P}^{-1}\hat{P}_{10}\hat{P}_{00}^{-1}) & -\hat{P}_{00}^{-1}\hat{P}_{01}\mathbb{P}^{-1}\\ -\mathbb{P}^{-1}\hat{P}_{10}\hat{P}_{00}^{-1} &  \mathbb{P}^{-1}\end{array}\right):{\mathcal R}\oplus {\mathcal R}^{\perp}\rightarrow{\mathcal K}^{\perp}\oplus {\mathcal K}.\]
	Conversely if $\hat{P}_{11}-\hat{P}_{10}\hat{P}_{00}^{-1}\hat{P}_{01}=0$ then $\hat{P}(\sigma,s,m)$ has a non trivial kernel and an element of the kernel is given by $(-\hat{P}_{00}^{-1}\hat{P}_{01}r,r)\in{\mathcal K}^{\perp}\oplus {\mathcal K}$. We will write in the following $m^2=s^2m_0^2$.
	\subsection{Computation of $\mathbb{P}$}
	\subsubsection{Computation of $\hat{P}_{11}$}
	 We have 
	\[\<r^*,\hat{P}(\sigma,s,m)r\>=\<r^*,\hat{P}(\sigma,s,0)r\>+s^2m_0^2\<r^*,r\>.\]
	We then compute: 
	\begin{eqnarray*}
		\big\<r^*,\hat{P}(\sigma,s,0)r\big\>&=&\left\<r^*,\left(-c^2(\sigma+sV)^2+\frac{i}{r}\p_r r^2\nu(\sigma+sV)\frac{1}{r}\right)r\right\>\\
		&=&c^2\Big[-\sigma^2\<r^*,r\>-2s\sigma\<r^*,Vr\>-s^2\<r^*,V^2r\>\Big]+i\left\<r^*,\frac{1}{r}\p_rr^2\nu(\sigma+sV)\right\>.
	\end{eqnarray*}
	Let us compute the different terms:
	\begin{eqnarray*}
		\<r^*,r\>&=&\int_{r_-}^{r_+}t^2\mathrm{d}t=\frac{r_+^3-r_-^3}{3},\\
		\<r^*,Vr\>&=&\<r^*,1\>=\int_{r_-}^{r_+}t \mathrm{d}t=\frac{r_+^2-r_-^2}{2},\\
		\<r^*,V^2r\>&=&\left\<r^*,\frac{1}{r}\right\>=\int_{r_-}^{r_+}\mathrm{d}t=r_+-r_-,\\
		\left\<r^*,\frac{1}{r}\p_rr^2\nu(\sigma+sV)\right\>&=&\int_{r_-}^{r_+}\p_tt^2\nu(\sigma+sV)\mathrm{d}t=r_+^2\nu(r_+)(\sigma+sV(r_{+}))-r_-^2\nu(r_-)(\sigma+sV(r_{-}))\\
		&=&-\sigma(r_+^2+r_-^2)-s(r_++r_-).
	\end{eqnarray*}
	\subsubsection{Computation of $\hat{P}_{10}\hat{P}_{00}^{-1}\hat{P}_{01}$}
	\label{L_01, L_10}
	Let 
	\begin{eqnarray*}
	\tilde{\sigma}&:=&\frac{r_++r_-}{r_+^2+r_-^2},\\
	S&:=&\frac{i}{r}\left[\left(\tilde{\sigma}-\frac{1}{r}\right)r^2\nu\p_r+\p_rr^2\nu\left(\tilde{\sigma}-\frac{1}{r}\right)\right]\frac{1}{r}.
	\end{eqnarray*} 
	We will assume that $\sigma=-s\tilde{\sigma}+\CO(s^2)$. We first consider the operator $\hat{P}_{10}:{\mathcal K}^{\perp}\rightarrow{\mathcal R}^{\perp}$. Let $f\in {\mathcal K}^{\perp}$; we have:
	\begin{eqnarray*}
		\lefteqn{\big\<r^*,\hat{P}(\sigma,s,m) f\big\>}\\
		&=&\left\<r^*,\left(-c^2(\sigma+sV)^2+\frac{i}{r}\p_r(\sigma+sV)\nu r+\frac{i}{r}(\sigma+sV)r^2\nu\pa_r\frac{1}{r}+m^2\right) f\right\>\\
		&=&\CO(\vert s\vert).
	\end{eqnarray*}
	We next consider the operator $\hat{P}_{01}:{\mathcal K}\rightarrow{\mathcal R}$. Let $f\in {\mathcal R}$. We find in the same way 
	\begin{eqnarray*}
		\big\<f,\hat{P}(\sigma,s,m) r\big\>&=&\CO(\vert s\vert).
		\end{eqnarray*}
		
	\begin{lemma}
		\label{lemma8}
		We have $ Sr\in\cR$. 
		\end{lemma}
	\begin{proof}
		We compute 
		\begin{eqnarray*}
		\<r^*,Sr\>
		&=&i\left\<1_{(r_-,r_+)},\p_rr^2\nu\left(\frac{r_++r_-}{r_+^2+r_-^2}-\frac{1}{r}\right)\right\>\\
		&=&i\left(-r_+^2\left(\frac{r_++r_-}{r_+^2+r_-^2}-\frac{1}{r_+}\right)-r_-^2\left(\frac{r_++r_-}{r_+^2+r_-^2}-\frac{1}{r_-}\right)\right)\\
		&=&-i(r_++r_--r_+-r_-)=0. 
		\end{eqnarray*}
		This completes the proof.
	\end{proof}

Now recall that $\hat{P}_{00}^{-1}(\sigma,s,m)=\hat{P}_{00}^{-1}(0,0,0)+\CO(\vert s \vert).$ Therefore 
	\begin{equation*}
	\hat{P}_{10}(\sigma,s,m)\hat{P}_{00}^{-1}(\sigma,s,m)\hat{P}_{01}(\sigma,s,m)=\hat{P}_{10}(\sigma,s,m)\hat{P}_{00}^{-1}(0,0,0)\hat{P}_{01}(\sigma,s,m)+\CO(\vert s\vert^3). 
	\end{equation*} 
	Now observe that
	\begin{eqnarray*}
	\<r^*,\hat{P}_{10}(\sigma,s,m)\hat{P}^{-1}(0,0,0)\hat{P}_{01}(\sigma,s,m)r\>=s^2\<r^*,S\hat{P}^{-1}(0,0,0)Sr\>+\CO(\vert s \vert^3). 
	\end{eqnarray*} 
       \subsection{Zeros of the principal term}
	\label{Zeros subsection}
	Let us put 
	\[\mathbb{P}_0(\sigma,s,m):=\hat{P}_{11}(\sigma,s,m)-s^2\<r^*,S\hat{P}^{-1}(0,0,0)Sr\>.\]
	Note that $\mathbb{P}_0(\sigma,s,m)=\mathbb{P}(\sigma,s,m)+\CO(\vert s\vert^3).$
	We have $\mathbb{P}_0(\sigma,s,m)=0$ if and only if
	\begin{eqnarray}
	\label{zeros}
		 \lefteqn{\sigma^2+3\frac{c^2s(r_+^2-r_-^2)+i(r_+^2+r_-^2)}{(r_+^3-r_-^3)c^2}\sigma}\nonumber\\&+\,3s\dfrac{sc^2(r_+-r_-)+s\big\langle r^*,S\hat{P}^{-1}(0,0,0)Sr\big\rangle+i(r_++r_-)}{(r_+^3-r_-^3)c^2}-s^2\dfrac{m_0^2}{c^2}=0.
	\end{eqnarray}
	
		\begin{remark}
		If $s=m=0$ then $\mathbb{P}_0(\sigma,0,0)=0$ if and only if $\sigma=0$ or $\sigma=-i\frac{3(r_+^2+r_-^2)}{c^2(r_+^3-r_-^3)}$. 
	\end{remark}
	\subsubsection{Computation of the roots}
	Let
	\begin{eqnarray*}
	a_1\!\!\!\!&:=&\!\!\!\!\frac{c^2(r_+^2-r_-^2)}{K},\quad a_2:=\frac{r_+^2+r_-^2}{K},\quad b_1:=\frac{c^2(r_+-r_-)+\big\langle r^*,S\hat{P}^{-1}(0,0,0)Sr\big\rangle}{K}-\frac{m_0^2}{c^2},\\
	 b_2\!\!\!\!&:=&\!\!\!\!\frac{r_++r_-}{K},\quad
	K:=\frac{c^2(r_+^3-r_-^3)}{3},\quad A:=sa_1+ia_2,\quad B:=s^2b_1+isb_2.  
	\end{eqnarray*}
	Then \eqref{zeros} writes $\sigma^2+A\sigma+B=0$. We have the roots 
	\[\sigma_{\pm}=\frac{A}{2}\left(-1\pm\sqrt{1-\frac{4B}{A^2}}\,\right)\]
	and are mainly interested in
	\begin{equation*}
	\sigma_+=\frac{A}{2}\left(-1+\sqrt{1-\frac{4B}{A^2}}\,\right)=-\frac{B}{A^3}(A^2+B)+\CO(\vert s\vert^3). 
	\end{equation*}
	The root $\sigma_-$ is a perturbation of $\sigma=-i\frac{3(r_+^2+r_-^2)}{c^2(r_+^3-r_-^3)}$ and is thus expected to stay in the lower half plane for small perturbations. 
	We compute 
	\begin{eqnarray*}
	-\frac{B}{A^3}&=&-(s^2a_1^2+a_2^2)^{-3}(s^2b_1+isb_2)(s^3a_1^3-3is^2a_1^2a_2-3sa_1a_2^2+ia_2^3)\\
	&=&(s^2a_1^2+a_2^2)^{-3}(3is^2b_2a_1a_2^2-is^2b_1a_2^3+sb_2a_2^3)+\CO(\vert s\vert^3),\\
	A^2+B&=&s^2a_1^2+2isa_1a_2-a_2^2+s^2b_1+isb_2. 
	\end{eqnarray*}
	Thus 
	\begin{eqnarray*}
	\sigma_+&=&(s^2a_1^2+a_2^2)^{-3}\big[-sb_2a_2^5+is^2a_2^3(b_1a_2^2+b_2^2-b_2a_1a_2)\big]+\CO(\vert s\vert^3)\\
	&=&-s\frac{b_2}{a_2}+is^2\frac{b_1a_2^2+b_2^2-b_2a_1a_2}{a_2^3}+\CO(\vert s\vert^3).
	\end{eqnarray*}
	We are interested in the sign of ${\rm Im}\, \sigma_+.$ We therefore compute 
	\begin{eqnarray*}
	\lefteqn{b_1a_2^2+b_2^2-b_2a_1a_2}\\
	&=&\frac{1}{K^3}\left(c^2(r_+-r_-)(r_+^2+r_-^2)^2+\frac{c^2(r_++r_-)^2(r_+^3-r_-^3)}{3}-(r_++r_-)c^2(r_+^2-r_-^2)(r_+^2+r_-^2)\right)\\
	&&+\frac{(r_+^2+r_-^2)^2\big\langle r^*,S\tilde{P}^{-1}_{0}(0)Sr\big\rangle}{K^3}-m_0^2\frac{(r_+^2+r_-^2)^2(r_+^3-r_-^3)}{3K^3}\\
	&=&\frac{c^2(r_+-r_-)^3}{3K^3}((r_+-r_-)^2+r_+r_-)+\frac{(r_+^2+r_-^2)^2\big\langle r^*,S\hat{P}^{-1}(0,0,0)Sr\big\rangle}{K^3}\\
	&-&m_0^2\frac{(r_+^2+r_-^2)^2(r_+^3-r_-^3)}{3K^3}.  
	\end{eqnarray*}
	Putting all together we find 
	\begin{align}
	\label{sigma+}
	\sigma_+=&-s\tilde{\sigma}+\frac{is^2c^2(r_+-r_-)^3((r_+-r_-)^2+r_+r_-)}{3(r_+^2+r_-^2)^3}\nonumber\\
	&+ \frac{is^2}{r_+^2+r_-^2}\left(\<r^*,S\hat{P}^{-1}(0,0,0)Sr\>-\frac{m_0^2(r_+^2-r_-^3)}{3}\right)+\CO(\vert s\vert^3).
	\end{align}
	\subsubsection{Computation of $\<r^*,S\hat{P}^{-1}(0,0,0)Sr\>$}
	We consider the equation $\hat{P}(0,0,0)f=Sr$:
	\begin{align*}
	-\frac{1}{r}\p_rr^2\mu\p_r\frac{1}{r}f&=\frac{i}{r}\p_rr^2\nu\left(\tilde{\sigma}-V\right).
	\end{align*}
	Multiplying by $-r$ then integrating over $(r_-,r)$ for some $r\in(r_-,r_+)$ yields
	\begin{align*}
	r^2\mu\p_r\frac{1}{r}f&=-i\Big(r^2\nu\left(\tilde{\sigma}-V\right)-r_-^2\nu(r_-)\left(\tilde{\sigma}-V(r_-)\right)\Big).
	\end{align*}
	Notice that we can indifferently use the value at $r_-$ or $r_+$ for the right-hand side \textit{cf.} the proof of Lemma \ref{lemma8}. It follows
	\begin{align*}
	\p_r\frac{1}{r}f&=-\frac{i}{r^2\mu}\left(r^2\nu\left(\tilde{\sigma}-\frac{1}{r}\right)-r_-^2\nu(r_-)\left(\tilde{\sigma}-\frac{1}{r_-}\right)\right)
	\end{align*}
	which is smooth at $r_\pm$ as $\mu(r)=(r-r_\pm)g(r)$ where $g$ is smooth and does not vanish at $r_\pm$. We eventually obtain
	\begin{align*}
	f(r)&=\frac{f(r_-)}{r_-}r-ir\int_{r_-}^{r}\left(\varrho^2\nu(\varrho)\left(\tilde{\sigma}-\frac{1}{\varrho}\right)-r_-^2\nu(r_-)\left(\tilde{\sigma}-\frac{1}{r_-}\right)\right)\frac{\mathrm{d}\varrho}{\varrho^2\mu(\varrho)},\quad\forall r\in(r_-,r_+).
	\end{align*}
	Then
	
	\begin{eqnarray*}
	\lefteqn{\big\langle r^*,S\hat{P}^{-1}(0,0,0)Sr\big\rangle}\\
	&=&\int_{r_-}^{r_+}r\frac{i}{r}\left[\left(\tilde{\sigma}-\frac{1}{r}\right)r^2\nu\p_r+\p_rr^2\nu\left(\tilde{\sigma}-\frac{1}{r}\right)\right]\frac{1}{r}f(r)\mathrm{d}r\\
	&=&i\int_{r_-}^{r_+}\left(\tilde{\sigma}-\frac{1}{r}\right)r^2\nu\p_r\frac{1}{r}f(r)\mathrm{d}r\\
	&+&i\left[r_{+}^2\nu(r_+)\left(\tilde{\sigma}-\frac{1}{r_+}\right)\frac{f(r_+)}{r_+}-r_{-}^2\nu(r_-)\left(\tilde{\sigma}-\frac{1}{r_-}\right)\frac{f(r_-)}{r_-}\right]\\
	&=&\int_{r_-}^{r_+}\left(\tilde{\sigma}-\frac{1}{r}\right)r^2\nu\left(r^2\nu(r)\left(\tilde{\sigma}-\frac{1}{r}\right)-r_-^2\nu(r_-)\left(\tilde{\sigma}-\frac{1}{r_-}\right)\right)\frac{\mathrm{d}r}{r^2\mu(r)}\\
	&+&\left(\tilde{\sigma}-\frac{1}{r_-}\right)r_-^2\nu(r_-)\int_{r_-}^{r_+}\left(r^2\nu(r)\left(\tilde{\sigma}-\frac{1}{r}\right)-r_-^2\nu(r_-)\left(\tilde{\sigma}-\frac{1}{r_-}\right)\right)\frac{\mathrm{d}r}{r^2\mu(r)}.
	\end{eqnarray*}
	Here we have used that $r_+^2\nu(r_+)\left(\tilde{\sigma}-\frac{1}{r_+}\right)=r_-^2\nu(r_-)\left(\tilde{\sigma}-\frac{1}{r_-}\right)$, see the proof of Lemma \ref{lemma8}. Putting
	\begin{align*}
	\Sigma(r)&:=\left(\tilde{\sigma}-\frac{1}{r}\right)r^2\nu(r),
	\end{align*}
	we see that
	\begin{align*}
	\big\langle r^*,S\hat{P}^{-1}(0,0,0)Sr\big\rangle&=\int_{r_-}^{r_+}\Big(\Sigma(r)^2-\Sigma(r_-)^2\Big)\frac{\mathrm{d}r}{r^2\mu(r)}.
	\end{align*}
	Using that $\nu(r)^2=1-\mu c^2$, we can write
	\begin{align}
	\label{Sigma(r)}
	\Sigma(r)^2-\Sigma(r_-)^2&=\tilde{\sigma}^2\big(r^4\nu(r)^2-r_-^4\nu(r_-)^2\big)-2\tilde{\sigma}\big(r^3\nu(r)^2-r_-^3\nu(r_-)^2\big)+\big(r^2\nu(r)^2-r_-^2\nu(r_-)^2\big)\nonumber\\
	&=\Big[\tilde{\sigma}^2\big(r^4-r_-^4\big)-2\tilde{\sigma}\big(r^3-r_-^3\big)+\big(r^2-r_-^2\big)\Big]-r^2\mu c^2(\tilde{\sigma}r-1)^2.
	\end{align}
	We may notice here that $\Sigma(r_+)^2-\Sigma(r_-)^2=0$ (because $\nu(r_-)=1=-\nu(r_+)$).
	
	The second term on the right-hand side of \eqref{Sigma(r)} can be directly integrated:
	\begin{align*}
	\int_{r_-}^{r_+}-r^2\mu c^2(\tilde{\sigma}r-1)^2\frac{\mathrm{d}r}{r^2\mu}&=-c^2\int_{r_-}^{r_+}\Big[\tilde{\sigma}^2r^2-2\tilde{\sigma}r+1\Big]\mathrm{d}r\\
	&=-c^2\left(\tilde{\sigma}^2\langle r^*,r\rangle-2\tilde{\sigma}\langle r^*,Vr\rangle+\langle r^*,V^2r\rangle\right)\\
	&=-c^2\frac{(r_+-r_-)^3(r_+^2-r_+r_-+r_-^2)}{3(r_+^2+r_-^2)^2}.
	\end{align*}
	\subsubsection{Sign of the imaginary part}
	Recall from \eqref{sigma+} the explicit form of $\sigma_+$. It follows that 
	\[{\rm Im}\, \sigma_+=\frac{s^2}{r_+^2+r_-^2}\int_{r_-}^{r_+}F(r)\frac{\mathrm{d}r}{r^2\mu}-\frac{s^2m_0^2(r_+^3-r_-^3)}{3(r_+^2+r_-^2)}+\cO(|s|^3),\]
	where $F(r):=\tilde{\sigma}^2\big(r^4-r_-^4\big)-2\tilde{\sigma}\big(r^3-r_-^3\big)+\big(r^2-r_-^2\big)$.
        We have the following factorization for $F(r)$ :
        \begin{align*}
		F(r)&=r^2(\tilde{\sigma}r-1)^2-r_-^2(\tilde{\sigma}r_--1)^2\\
		&=\Big[\tilde{\sigma}r^2-r+r_-\big(\tilde{\sigma}r_--1\big)\Big]\Big[\tilde{\sigma}r^2-r-r_-\big(\tilde{\sigma}r_--1\big)\Big]\\
		&=\left[\tilde{\sigma}(r-r_+)\left(r+\frac{r_-(r_+-r_-)}{\tilde{\sigma}(r_-^2+r_+^2)}\right)\right]\left[\tilde{\sigma}(r-r_-)\left(r-\left(\frac{1}{\tilde{\sigma}}-r_-\right)\right)\right]\\
		&=\tilde{\sigma}^2(r-r_-)(r-r_+)\left(r+\frac{r_-(r_+-r_-)}{r_-+r_+}\right)\left(r-\frac{r_+(r_+-r_-)}{r_-+r_+}\right).
		\end{align*}
        Writing $\mu(r)=\frac{\Lambda}{3r^2}(r-r_-)(r_+-r)(r-r_n)(r-r_c)$. Let $\beta=\frac{r_+-r_-}{r_++r_-}$.
        Note that $\beta r_+<r_+$ and that $r_-<\beta r_+$ if and only if 
        \begin{equation}
        \label{C1}
        r_+^2-2r_+r_--r_-^2>0.
        \end{equation}
        We will suppose \eqref{C1} in the following. We then have : 
        \begin{eqnarray*}
        \frac{1}{s^2}{\rm Im}\sigma_+&=&-\frac{3\tilde{\sigma}^2}{\Lambda(r_+^2+r_-^2)}\left(\int_{r_-}^{\beta r_+}\frac{(r+\beta r_-)(r-\beta r_+)}{(r-r_n)(r-r_c)}\mathrm{d}r+\int_{\beta r_+}^{r_+}\frac{(r+\beta r_-)(r-\beta r_+)}{(r-r_n)(r-r_c)}\mathrm{d}r\right)\\
        &-&\frac{m_0^2(r_+^3-r_-^3)}{3(r_+^2+r_-^2)}\\
        &\ge&-\frac{3\tilde{\sigma}^2}{\Lambda(r_+^2+r_-^2)(\beta r_+-r_n)(\beta r_+-r_c)}\int_{r_-}^{r_+}(r+\beta r_-)(r-\beta r_+)\mathrm{d}r-\frac{m_0^2(r_+^3-r_-^3)}{3(r_+^2+r_-^2)}.
        \end{eqnarray*}
        An explicit calculation gives 
        \[-\int_{r_-}^{r_+}(r+\beta r_-)(r-\beta r_+)\mathrm{d}r=\frac{r_+-r_-}{6(r_++r_-)^2}(r_+^4+r_-^4-26r_+^2r_-^2).\]
        Thus if we suppose \eqref{C1} and
        \begin{eqnarray}
        \label{C2}
        &r_+^4+r_-^4-26r_+^2r_-^2>0,\\
        \label{Cm}
        &m_0^2<\dfrac{3(r_+^4+r_-^4-26r_+^2r_-^2)}{2\Lambda(r^2_++r^2_-)^2(\beta r_+-r_n)(\beta r_+-r_c)(r_+^2+r_-^2+r_+r_-)},
        \end{eqnarray}
        then
        \begin{equation}
        \label{posIm}
        {\rm Im}\, \sigma_+\ge C_1 s^2,\qquad C_1>0
        \end{equation}
        uniformly in $\vert s\vert\le s_0$ for $s_0$ small enough. 
        Noting that in the limit $\Lambda\rightarrow 0$, $r_+\rightarrow \infty$ whereas $r_-$ stays in a bounded domain, so that for $\Lambda$ sufficiently small \eqref{C1} and \eqref{C2} hold. We fix $\Lambda>0$ small enough and $m_0>0$ as in \eqref{Cm} in the following. 
        \begin{remark}
        \label{remLambdalarge}
        \begin{enumerate}
        \item If 
        \begin{equation}
        \label{eqneg}
        r_+^2-2r_+r_--r_-^2<0,
        \end{equation}
        then we have $\beta r_+<r_-$ and the function $F$ is negative on $(r_-,r_+)$. $\sigma_+$ has in this case negative imaginary part for $s$ small enough. We then expect that all modes have negative imaginary part for $s$ small enough.  \eqref{eqneg} is fulfilled near extremal black holes for $9\Lambda M^2$ close to $1$ and $Q$ close to zero.
        
        \item Planck's latest value for the cosmological constant \cite{PC} is $\Lambda=1.106\,10^{-52}\,\mathrm{m}^{-2}$. Replacing $M$ and $Q$ in natural units by respectively $GM/c^2$ and $Q(G/4\pi\varepsilon_0)^{1/2}$ with $G=6.674\,10^{-11}\,\mathrm{m}^3.\mathrm{kg}^{-1}.\mathrm{s}^{-2}$ the gravitational constant, $c=2.998\,10^{8}\,\mathrm{m}.\mathrm{s}^{-1}$ the speed of light in vacuum and $\varepsilon_0=8.854\,10^{-12}\,\mathrm{F.m}^{-1}$ the vacuum permittivity, numerical computations give 
        \[r_+^4+r_-^4-26r_+^2r_-^2>7.357\,10^{104}\,\mathrm{m}^4,\quad r_+^2-2r_+r_--r_-^2>2.714\,10^{52}\, \mathrm{m}^2\]
         for black hole masses ranging from $0$ to $1.31\,10^{41}\,\mathrm{kg}$, the mass of the heaviest known black hole \emph{TON 618} (\emph{cf.} \cite{CCFMNOS}), and electric charges $|Q|\leq\frac{3GM}{2\sqrt{2}c^2}$. Therefore our conditions are fulfilled in physical realistic cases. 
         \end{enumerate}
    	\end{remark}
	\subsection{Proof of the main theorem}
	\label{Proof of the main theorem}
	Let $\sigma\in \partial D(\sigma_+,\frac{C_1}{2}s^2)\subset \{\sigma\in \C\,:\, {\rm Im}\, \sigma>0\}$, where $C_1$ is as in  \eqref{posIm}. We then have 
	\[\vert \mathbb{P}_0(\sigma,s,m)\vert\gtrsim\vert (\sigma-\sigma_+)(\sigma-\sigma_-)\vert\gtrsim s^{2}.\]
	On the other hand by the computations in the previous subsections we have 
	\[\big\vert \mathbb{P}_0(\sigma,s,m)-\mathbb{P}(\sigma,s,m)\big\vert\lesssim \vert s\vert^{3}. \]
	The existence of a zero of $\mathbb{P}(\sigma,s,m)$ in $D(\sigma_+,\frac{C_1}{2}s^2)$ then follows from the Rouch\'e theorem for $s$ small enough. All the arguments in the proof work for the operators restricted to ${\rm Ker} \Delta_{\mathbb{S}^2_{\omega}}$. There is therefore an element in ${\rm Ker}\, \hat{P}(\sigma,s,m)$ that lies also in  ${\rm Ker} \Delta_{\mathbb{S}^2_{\omega}}$ for some $\sigma$ with ${\rm Im}\, \sigma\geq c_1 s^2>0$. This completes the proof.
	\qed
	
	\section{Growing modes on the De Sitter-Kerr-Newman metric}
	\label{Sec5}
	We extend in this last section Theorem \ref{MainThm} to the De Sitter-Kerr-Newman metric. Let
	\begin{equation*}
	  \mu_a(r):=(r^{2}+a^{2})\left(1-\frac{\Lambda r^{2}}{3}\right)-2Mr+\left(1+\frac{\Lambda a^{2}}{3}\right)^{2}Q^{2}
	\end{equation*}
	As $\mu_a=r^2\mu+\CO(a^2)$, there exists $a_1>0$ such that $\mu_a$ still has two largest positive roots $r_{-,a}<r_{+,a}$ for $a\in(-a_1,a_1)$. Note that $r_{\pm, a}$ are close to $r_{\pm}$ for $\vert a\vert$ small. On $\mathcal{M}_a=\mathbb{R}_t\times(r_{-,a},r_{+,a})_r\times\mathbb{S}^2_{\theta,\phi}$ we define
	\begin{equation*}
	g_a=\frac{\mu_a}{(1+\lambda)^{2}\rho^{2}}\big(\mathrm{d}t-a\sin^{2}\theta\mathrm{d}\phi\big)^{2}-\frac{\kappa\sin^{2}\theta}{(1+\lambda)^{2}\rho^{2}}\big(a\mathrm{d}t-(r^{2}+a^{2})\mathrm{d}\phi\big)^{2}-\rho^{2}\left(\frac{\mathrm{d}r^{2}}{\mu_a}+\frac{\mathrm{d}\theta^{2}}{\kappa}\right)
	\end{equation*}
	where
	\begin{equation*}
	\lambda=\frac{\Lambda a^{2}}{3},\qquad\qquad\kappa=1+\lambda\cos^{2}\theta,\qquad\qquad\rho^{2}=r^{2}+a^{2}\cos^{2}\theta.
	\end{equation*}
	Observe that $g_a=g+\CO(|a|)$ and $\rho=r+\CO(a^2)$. Following \cite[Lemma 3.4]{HiVa}, we introduce for $\vert a\vert$ small new coordinates $t_*:=t-T(r),\,\phi_*:=\phi-\Phi(r)$ with $T,\,\Phi$ smooth on $(r_{-,a},r_{+,a})$, such that the inverse metric reads near $r_{\pm,a}$ 
	\begin{align*}
		g_a^{-1}=&-\frac{\mu_a}{\rho^2}\big(\p_r\mp c_{\pm,a}\p_{t_*}\mp\tilde{c}_{\pm,a}\p_{\phi_*}\big)^2 \pm \frac{2a(1+\lambda)}{\rho^2}\big(\p_r\mp c_{\pm,a}\p_{t_*}\mp\tilde{c}_{\pm,a}\p_{\phi_*}\big)\p_{\phi_*}\\
		&\pm \frac{2(1+\lambda)}{\rho^2}(r^2+a^2)\big(\p_r\mp c_{\pm,a}\p_{t_*}\mp\tilde{c}_{\pm,a}\p_{\phi_*}\big)\p_{t_*}\\
		&-\frac{(1+\lambda)^2}{\kappa\rho^2\sin^2\theta}\big(a\sin^2\theta\p_{t_*}+\p_{\phi_*}\big)^2-\frac{\kappa}{\rho^2}\p_{\theta}^2.
	\end{align*}
	Here $\tilde{c}_{\pm,a}=\frac{a}{r^2+a^2}c_{\pm,a}$ and $\mu_ac_{\pm,a}=r^2\mu c_{\pm,0}+\CO(a^2)$. The electromagnetic interaction generated by the charge and the rotation of the black hole is now encoded by the electromagnetic potential $A_a:=\frac{Q r}{\rho^{2}}\left(\mathrm{d}t-a\sin^{2}\theta\mathrm{d}\phi\right)$ in Boyer-Lindquist coordinates; it becomes in the new coordinates:
	\begin{equation*}
	A_a=\frac{Q r}{\rho^2}\mathrm{d}t_*\pm \left(\frac{Qr}{r^2+a^2}c_{\pm,a}+\frac{Q(1+\lambda)r}{\mu}\right)\mathrm{d}r-\frac{Qr}{\rho^2}a\sin^2\theta \mathrm{d}\phi_*. 
	\end{equation*}
	We write $A_a=A_a^{t_*}(r,\theta) \mathrm{d}t_*+ A^r_a(r) \mathrm{d}r+A^{\phi_*}_a(r,\theta)\mathrm{d}\phi_*$ and put $R_a'(r)=A_a^r(r)$. 
	After a tedious computation, we see that the Fourier transformed operator associated to the gauge-transformed Klein-Gordon operator $\rho e^{iR_a}e^{i\sigma t_*}\big(\,\cancel{\Box}_{g_a}+m^{2}\big)e^{-i\sigma t_*}e^{-iR_a}\frac{1}{\rho}$ is
	\begin{align*}
	\hat{P}(\sigma,s,m,a)=&-g_a^{t_*t_*}(\sigma+sV_a)^2-\frac{i}{\rho}\left((\sigma+sV_a)\rho^2g_a^{t_*r}\p_{r}+\p_{r}\rho^2g_a^{t_*r}(\sigma+sV)\right)\frac{1}{\rho}\\
	&-2ig_a^{t_*\phi_*}(\sigma+sV_a)(\p_{\phi_*}+iasV_a\sin^2\theta)-\frac{1}{\rho}\p_{r}\mu_a\p_{r}\frac{1}{\rho}\\
	&+\frac{1}{\rho}\left(\p_{r}\rho^2g_a^{r\phi_*}(\p_{\phi_*}+iasV_a\sin^2\theta)+(\p_{\phi_*}+iasV_a\sin^2\theta)\rho^2g_a^{r\phi_*}\p_{r}\right)\frac{1}{\rho}\\
	&-\frac{1}{\rho\sin\theta}\p_{\theta}\kappa\sin\theta \p_{\theta}\frac{1}{\rho}+g_a^{\phi_*\phi_*}(\p_{\phi_*}+iasV_a\sin^2\theta)^2+m^2,
	\end{align*}
	where $V_a=\frac{\rho^2}{r}=V+\CO(a^2)$ and $g_a^{t_*t_*},\,g_a^{t_*r},\,g_a^{t_*\phi_*}\,g_a^{r\phi_*},\,g_a^{\phi_*\phi_*}$ are the coefficients of the inverse metric, smooth on $Y$ and satisfying $g_a^{t_*t_*}=-c^2+\CO(a^2)$, $g_a^{t_*r}=-\nu+\CO(a^2)$, $g_a^{\phi_*\phi_*}=-r^{-2}\sin^{-2}\theta+\CO(a^2)$ and $g_a^{t_*\phi_*},\,g_a^{r\phi_*}=\CO(|a|)$. Theorem \ref{Fredholm} still holds for $a$ sufficiently small that is $\hat{P}(\sigma,s,m,a):\CX^{\gamma}_a\rightarrow\bar{H}^{\gamma-1}(Y)$ is Fredholm of index 0. We consider in the following $\hat{P}(\sigma,s,m,a)$ restricted to ${\rm Ker}(\p_{\phi_*})$, i.e. we put $\p_{\phi_*}=0$ in the above expression. With this restriction $\CX^{\gamma}_a$ becomes independent of $a$. The operator $-\frac{1}{\sin\theta}\p_{\theta}\kappa \sin\theta\p_{\theta}$ can in this context be understood as the restriction of 
	\[-\frac{1}{\sin\theta}\p_{\theta}\kappa \sin\theta\p_{\theta}-\frac{\kappa}{\sin^2\theta}\p^2_{\phi_*}=-\kappa\Delta_{\mathbb{S}^2_{\theta,\phi_*}}+\frac{2\Lambda a^2}{3}\sin\theta\cos\theta\p_{\theta}\]
	to ${\rm Ker}(\p_{\phi_*})$ which is smooth. For $s$ sufficiently small and $m^2\le c_0 \vert q\vert^2$ we obtain using Theorem \ref{MainThm} the existence of $\sigma(s,m)$ with ${\rm Im}\, \sigma(s,m)\ge \frac{C_1}{2}s^2$ and
	\[{\rm Ker} (\hat{P}(\sigma(s,m),s,m,0))\neq \{0\}.\]
	We now write $a^2=a_0^2s^2$ and observe that    
	\begin{eqnarray*}
	\hat{P}_{ij}(\sigma,s,m,a)=\hat{P}_{ij}(\sigma,s,m,0)+s^2\CO(a_0^2),\qquad\forall i,j\in \{0,1\} 
	\end{eqnarray*} 
         uniformly in $\vert \sigma-\sigma(s,m)\vert \le \frac{C_1}{4} s^2$. In particular we have 
         \[\hat{P}_{00}(\sigma,s,m,a)=\hat{P}_{00}(\sigma,s,m,0)({\bf 1}+\hat{P}^{-1}_{00}(\sigma,s,m,0)s^2\CO(a_0^2)).\]
         For $a$ sufficiently small $\hat{P}_{00}(\sigma,s,m,a)$ is thus invertible and we have 
         \[\hat{P}_{00}^{-1}(\sigma,s,m,a)=\hat{P}_{00}^{-1}(\sigma,s,m,0)+s^2\CO(a_0^2)\]
         uniformly in $\vert \sigma-\sigma(s,m)\vert\le \frac{C_1}{4}s^2$. 
         Putting all together we find 
         \[\mathbb{P}(\sigma,s,m,a)=\mathbb{P}(\sigma,s,m,0)+s^2\CO(a_0^2)\]
         uniformly in $\vert \sigma-\sigma(s,m)\vert\le \frac{C_1}{4} s^2$. Then for $s$ sufficiently small we have :
         \begin{eqnarray*}
         \vert \mathbb{P}(\sigma,s,m,0)\vert \ge C_2 s^2,\quad \vert \mathbb{P}(\sigma,s,m,a)-\mathbb{P}(\sigma,s,m,0)\vert\le C_3a_0^2s^2
         \end{eqnarray*}
         for some constants $C_2>0,\, C_3>0$. By choosing $a_0^2< \frac{C_2}{C_3}$ we obtain by the Rouch\'e theorem:
         \begin{theorem}
		\label{MainThmDsKN}
		Fix $M>0$ and $Q$ such that $0<|Q|<\frac{3M}{2\sqrt{2}}$. There exists $\Lambda_0>0$ such that for all $\Lambda\in(0,\Lambda_0)$ there exists a constant $c_0(\Lambda)>0$ such that the following holds. Suppose $m^2\le c_0(\Lambda)q^2,\, a^2\le c_0(\Lambda)q^2$. Then for $|q|>0$ sufficiently small, there exists $\sigma\in \C$ with ${\rm Im}\, \sigma\geq c_1 q^2>0$ such that 
		\[{\rm Ker}\,\hat{P}(\sigma,s,m,a)\cap {\rm Ker}(\p_{\phi_*})\neq\{0\}.\]
	\end{theorem}
	\begin{remark}
		Let $u\in{\rm Ker}\,\hat{P}(\sigma,s,m)\cap {\rm Ker}(\p_{\phi_*})$ so that $v(t_*,r,\theta)=e^{-i\sigma t_*}\frac{e^{-iR_a}}{\rho}u(r,\theta)$ solves
		\[\left(\,\cancel{\Box}_{g_a}+m^{2}\right)v=0.\]
		Again 
		\[\CH(v)=\Vert(\p_{t_*}-iqA_{t_*})v\Vert^2+\Vert v\Vert^2_{\bar{H}^1}\] 
		is finite but exponentially growing in $t_*$. Going back to the original coordinates and gauge, we find 
		\[v(t,r,\theta)=e^{-i\sigma t}e^{i\sigma T(r)}\frac{e^{-iR(r)}}{r}u(r,\omega)=e^{-i\sigma t}w(r,\theta)\] 
		with \[w\in e^{(\epsilon-{\rm Im}\, \sigma)|T(r)|}L^2\left((r_{-,a},r_{+,a})\times\mathbb{S}^2,\frac{\varpi^2}{\mu_a\kappa}\mathrm{d}r\mathrm{d}\omega\right),\qquad\varpi^2=(r^2+a^2)^2\kappa-a^2\mu_a\sin^2\theta\]
		for all $\epsilon>0$ (notice that $\varpi^2>0$ if $a_0$ is small enough). Following \cite{GGH}, we associate to this solution the energy
		\[\dot{\CE}(v)=e^{2({\rm Im\,}\sigma)t}\left(\Vert (\sigma+k)w\Vert^2+(h_0w,w)\right),\]
		where $(.,.)$ is the scalar product in $L^2\left((r_{-,a},r_{+,a})\times\mathbb{S}^2,\frac{\varpi^2}{\mu_a\kappa}\mathrm{d}r\mathrm{d}\omega\right)$ and
		\begin{align*}
		k=&\ sV+\frac{a(\mu_{a}-(r^2+a^2)\kappa)}{\varpi^2}\left(D_\phi+sVa\sin^2\theta\right),\\
		h_0=&\ \frac{\rho^4\mu_{a}\kappa}{\varpi^4\sin^2\theta}(D_\phi+sVa\sin^2\theta)^2+\frac{\mu_{a}\kappa}{(1+\lambda)^2\varpi^2}D_r\mu_{a}D_r+\frac{\mu_{a}\kappa}{(1+\lambda)^2\varpi^2\sin\theta}D_\theta\sin\theta D_\theta\\
		&+\frac{\rho^2\mu_{a}\kappa}{(1+\lambda)^2\varpi^2}m^2.
		\end{align*}
		Here we used the notation $D_\bullet=-i\p_\bullet$. We have $h_0>0$ and thus $\dot{\CE}(v)>0$. Furthermore, this energy is finite for all $t>0$ but it is exponentially growing.
	\end{remark}
	

\begin{thebibliography}{99}
		\bibitem[Be]{Be}
		N. Besset, \emph{Decay of the Local Energy for the Charged Klein-Gordon Equation in the Exterior De Sitter-Reissner-Nordstr\"om Spacetime}, arXiv:1812.09390. 
		\bibitem[BoHa]{BoHa} 
		J.-F. Bony, D. H\"afner, \emph{Decay and non-decay of the local energy for the wave equation in the De Sitter - Schwarzschild metric},
		Comm. Math. Phys. {\bf 282} (2008), no. 3, 697-719.
		\bibitem[CCDHJ]{CCDHJ}
		V. Cardoso, J. Costa, K. Destounis, P. Hintz and Aron Jansen, \emph{Strong cosmic censorship in charged black-hole spacetimes: still subtle}, Phys. Rev. D {\bf 98} (2018), 104007.
		\bibitem[CCFMNOS]{CCFMNOS}
		E. Corbett, S. Croom, L. di Fabrizio, R. Maiolino, H. Netzer, E. Oliva and O. Shemmer, \emph{Near-Infrared Spectroscopy of High-Redshift Active Galactic Nuclei. I. A Metallicity-Accretion Rate Relationship}, American Astronomical Society {\bf 614} (2004), no. 2.
		\bibitem[Dy]{Dy}
		S. Dyatlov, \emph{Quasi-normal modes and exponential energy decay for the Kerr-de Sitter black hole}, Comm. Math. Phys. {\bf 306} (2011), 119-163.
		\bibitem[GGH]{GGH}
		V. Georgescu, C. G\'{e}rard, D. H\"{a}fner, \emph{Asymptotic completeness for superradiant Klein-Gordon equations and applications to the De Sitter-Kerr metric}, J. Eur. Math. Soc. {\bf 19} (2017), 2171-��2244. 
		\bibitem[HHV]{HHV}
		D. H\"afner, P. Hintz, A. Vasy, \emph{Linear stability of slowly rotating Kerr black holes}, arXiv:1906.00860.
		\bibitem[Hi]{HinlsKN}P. Hintz, \emph{Nonlinear stability of the Kerr-Newman-de Sitter family of charged black holes}, Annals of PDE {\bf 4} (2018), 11.
		\bibitem[HiVa]{HiVa}
		P. Hintz, A. Vasy, \emph{The global non-linear stability of the Kerr-de Sitter family of black holes}, Acta Mathematica {\bf 220} (2018), 1-206. 
		\bibitem[KlSz]{KlSz}
		S. Klainerman, J. Szeftel, \emph{Global Nonlinear Stability of Schwarzschild Spacetime under Polarized Perturbations}, arXiv:1711.07597. 
		\bibitem[Mo]{Mo}
		M. Mokdad, \emph{Reissner-Nordstr\"{o}m-de Sitter Manifold: Photon Sphere and	Maximal Analytic Extension}, Class. Quantum Grav. \textbf{34}(2017), 175014.
		\bibitem[Mosch]{Mosch}
		G. Moschidis,  \emph{Superradiant instabilities for short-range non-negative potentials on Kerr spacetimes and applications}, Journ. Funct. Anal. {\bf 273}, 2719-2813.
		\bibitem[PC]{PC}
		Planck Collaboration, \emph{Planck 2018 results. VI. Cosmological parameters}, arXiv:1807.06209.
		\bibitem[SR]{SR}
		Y. Shlapentokh-Rothman, \emph{Exponentially growing finite energy solutions for the Klein-Gordon equation on sub-extremal Kerr-spacetimes}, Comm. Math. Phys. {\bf 329} (2014), 859-891. 
		\bibitem[Va]{Va} 
		A. Vasy, \emph{ Microlocal analysis of asymptotically hyperbolic and Kerr-de Sitter spaces (with an appendix by Semyon Dyatlov)},
		Inventiones Math {\bf 194} (2013), 381-513.
		\bibitem[Zw]{Zw}
		M. Zworski, \emph{Resonances for asymptotically hyperbolic manifolds: Vasy's method revisited},  Journal of Spectral Theory {\bf 6}  (2016), 1087--1114.
	\end{thebibliography}
\end{document}